\documentclass{article}

\usepackage{tikz}
\usepackage[usenames,dvipsnames]{pstricks}
\usepackage{epsfig}
\usepackage{pst-grad} 
\usepackage{pst-plot} 
\usepackage[space]{grffile} 
\usepackage{etoolbox} 
\makeatletter 
\patchcmd\Gread@eps{\@inputcheck#1 }{\@inputcheck"#1"\relax}{}{}
\makeatother

\usepackage[english]{babel}

\usepackage{geometry}
\usepackage{amsmath,amsfonts,amssymb,amsthm}
\usepackage{graphicx}
\usepackage[colorlinks=true, allcolors=blue]{hyperref}
\usepackage{comment}
\usepackage{breqn}
\usepackage{caption}
\usepackage{subcaption}

\newtheorem{definition}{Definition}
\newtheorem{theorem}[definition]{Theorem}

\newtheorem{proposition}[definition]{Proposition}

\newcommand{\Z}{\mathbb{Z}}

\newcommand{\calb}{\mathcal{B}}

\title{A characterization of locally ordered ternary relations in terms of digraphs}
\author{
Guillermo Gamboa Quintero \thanks{Computer Science Institute of Charles University, Prague, Czechia. Supported by Charles University Project PRIMUS/24/SCI/012 and European Union’s Horizon 2020 research and innovation programme under the Marie Skłodowska-Curie grant agreement No 101007705. \textit{E-mail}: gamboa@iuuk.mff.cuni.cz} 
\and 
Mart\'{i}n Matamala \thanks{DIM-CMM, CNRS-IRL2807, Universidad de Chile, Chile. Supported by ANID  Basal program FB210005.  \\ \textit{E-mail}: mar.mat.vas@dim.uchile.cl} 
\and 
Juan Pablo Pe\~{n}a \thanks{Departamento de Ingenier\'{i}a Matem\'{a}tica, Universidad de Chile, Chile. Supported by ANID Doctoral Fellowship grant 21211955. \textit{E-mail}: juan.pena@dim.uchile.cl}}

\date{}

\begin{document}

\maketitle

\begin{abstract}
In 1917, Huntington and Kline, followed by Huntington in 1924, studied systems of axioms for ternary relations aiming to capture the concepts of linear order (called \emph{betwenness}) and cycle order, respectively. Among many other properties, they proved there are several independent set of axioms defining either linear order or cycle order.


In this work, we consider systems arising from the four common axioms of linear order and cycle order, together with a new axiom \textbf{F}, stating that if a tuple $abc$ is in the relation, then either $cba$ or $bca$ is in the relation as well. Although, at first glace this allows for a much richer type of systems, we prove that these are either of linear order or cycle order type. Our main result is a complete characterization of the finite systems satisfying the set of axioms $\{ \textbf{B}, \textbf{C}, \textbf{D}, \textbf{F}, \textbf{2} \}$, where \textbf{B}, \textbf{C}, \textbf{D} and \textbf{2} are axioms presented by Huntington. Unlike what happens in the previous situation, with this modification we obtain a larger family of systems which we characterize in terms of digraphs.

\end{abstract}

\section{Introduction}

The concept of betweenness is a classical object in mathematics, originating in geometry and order theory.  Based on the works of Pasch \cite{pasch} and Peano \cite{peano}, Hilbert \cite{hilbert} described the concept of \textit{being in between} as a ternary relation $\calb$ over a ground set $V$. Following this, a fundamental problem was to understand when such a relation can be associated to a linear or total order on $V$. In 1917, Huntington and Kline \cite{huntkline} defined a betweenness as a ternary relation  $\calb$ on a set $V$ which satisfies at least one of eleven sets of axioms. Below we present one of these sets of axioms. To ease presentation, we have denoted by $abc$ the tuple $(a,b,c)$. The set consists of the following six axioms:
\begin{itemize}
    \item[\textbf{A}] For all $a,b,c\in V$, $abc\in\calb$ implies $cba\in\calb$.
    
    \item[\textbf{B}] For all $a,b,c\in V$, if $a,b$ and $c$ are pairwise distinct, then $\calb$ contains at least one of the triples $abc,acb,bac,bca,cab$ or $cba$.
    
    \item[\textbf{C}] For all $a,b,c\in V$, $abc\in\calb$ implies $acb \notin\calb$.
    
    \item[\textbf{D}]For all $a,b,c\in V$, $abc\in\calb$ implies $a,b$ and $c$ are pairwise distinct.
    
    \item[\textbf{1}] For all $a,b,c,x\in V$, $abc\text{ and } bcx\in\calb$ implies $abx\in\calb$.

    \item[\textbf{2}] For all $a,b,c,x\in V$, $abc\text{ and } bxc \in\calb$ implies $abx\in\calb$.
    
\end{itemize}
It is clear that, for any strict total order $<$ on $V$, the ternary relation  $\calb_<$ given by  $$abc\in\calb_< \iff a<b< c \lor c <b< a$$ satisfies axioms \textbf{A}, \textbf{B}, \textbf{C}, \textbf{D}, \textbf{1} and \textbf{2}. Conversely, in \cite{hunt1935}, it was proved that for any ternary relation $\calb$ satisfying these six axioms one can define a strict total order $<$ such that  $\calb=\calb_<$.


In what follows we still call a ternary relation satisfying axioms \textbf{A}, \textbf{B}, \textbf{C}, \textbf{D}, \textbf{1} and \textbf{2} a betweenness. Morever, we call a ternary relation $\calb$ on a set $V$ a \emph{pre-betweenness} if it only satisfies axioms
\textbf{A}, 
\textbf{B}, 
\textbf{C}, 
\textbf{D} and 
\textbf{2}. 
Although Huntington proved that the set of axioms \textbf{A}, 
\textbf{B}, 
\textbf{C}, 
\textbf{D}, 
\textbf{1}, 
\textbf{2} is independent, in \cite{shepperd1956transitivities}, Shepperd proved that a pre-betweenness is a betweenness unless the ground set $V$ has only four points and the relation $\calb$ is that defined by the metric space induced by a non-oriented cycle of length four. To illustrate our point, let $d: V\times V \rightarrow \mathbb{R}^+$ be a function. One may define a ternary relation  $\calb_d$ by $$abc\in\calb_d \Leftrightarrow d(a,c)=d(a,b)+d(b,c).$$ One of the earliest formal treatments of ternary relations  $\calb_d$  for a metric $d$ was given by Menger \cite{menger}. In this  context  $\calb_d$ encodes the idea of one point lying \textit{in between} two others. Moreover, Menger noticed that  $\calb_d$ satisfies \textbf{A}, \textbf{C}, \textbf{D},  \textbf{2} and \textbf{3}, where $$\textbf{3}: \forall a,b,c,x\in V, abc, bxc \in\calb\implies axc\in\calb.$$

This framework was further developed in subsequent axiomatic studies, such as  Altwegg’s investigation of partially ordered sets \cite{altwegg1950axiomatik} and Sholander’s work on trees, lattices, and order \cite{sholander1952trees}, both of which explored the interplay between order-theoretic properties and ternary relations. In this context, the result of Sheppard was rediscovered in \cite{richmondrichmond}, for metrics $d$ on a set $V$ where any 3-set of $V$ satisfies the triangle inequality as equality. In this case,  for such  a metric $d$, the relation  $\calb_d$ satisfies \textbf{B}. In \cite{beaudou2013lines} a ternary relation $\calb$ satisfying \textbf{A}, \textbf{C}, \textbf{D}, \textbf{2} and \textbf{3} was called a \textit{pseudometric betweenness} and one such that $\calb=\calb_d$, for some metric $d$, a  \textit{metric betweenness}. Moreover, it was proved that the 3-uniform hypergraph associated with certain pseudometric betweenness cannot be defined by a metric. One of the main ingredient was a technical lemma whose content was, essentially, the result of Sheppard. Beyond total orders, betweenness structures induced by partial and strict orders have also been investigated \cite{fishburn1971betweenness,Lih,rautenbach2011strict,bankston2013road,courcelle2020betweenness}.

Another context where the notion of being in between arises is that of \emph{transit functions}. These functions, originally studied in the context of graphs and posets \cite{mulder2007transit}, associate to every unordered pair ${x,y}$ a set of points \textit{in between} $x$ and $y$. More recent work has extended this to \emph{directed transit functions} over quasimetric spaces \cite{anil2025directed}, where notions of directionality and asymmetry play a crucial role. In these settings, the natural ternary structure is that induced by interval-like behavior.

In 1924, Huntington proposed an axiomatization, similar to that given above for betweennesses, to capture cyclic orders \cite{huntington1924sets}. The most familiar example is given by a set of points on the circumference of a circle, where $uvw$ belongs to the cyclic betweenness relation if $u$, $v$, and $w$ are distinct and appear in that order when traversing the circle clockwise, before completing a full turn. 
His axioms capture the core properties of such cyclically ordered triples and laid the foundation for later developments in the theory of circular orderings.
Interestingly, the proposed system of axioms consists of those of betweenness but replacing axioms \textbf{A} and \textbf{1} (in the system defined by $\{\textbf{A}, \textbf{B}, \textbf{C},\textbf{D},\textbf{1},\textbf{2}\}$) by the following axiom \textbf{E}:
$$\textbf{E}: \forall a,b,c\in V, abc\in\calb\implies bca\in\calb.$$
In both cases, betweenness and cycle order, the following dual version of axiom \textbf{C} holds 
$$\textbf{C'}: \forall a,b,c\in V, abc\in \calb \implies bac\notin \calb.$$
For finite sets $V$, the previously mentioned results of Huntington and Shepperd can be summarized by saying that the ternary relation $\calb$ on $V$ is  isomorphic to a non-oriented path or to a non-oriented cycle of length four, when $\calb$ is a pre-betweenness, and to a  directed cycle, when it is a cycle order. We shall see that the one dimensional aspect of these ternary relations can be captured with the concept of \emph{line}. For a ternary relation $\calb$ on a set $V$, the \textit{line} $\overrightarrow{ab}$ is the set of all points $x$ in $V$ such that $x=a, x=b$, or $xab, axb$ or $abx$ belong to $\calb$. Then, \[
\overrightarrow{ab} = \{ x \in V \mid xab \in\calb\text{ or } axb\in\calb\text{ or } abx \in\calb\}\cup \{a,b\}.
\]
It turns out that every pre-betweenness or cycle order defines a system with a unique line. This aligns with the intuition coming from linear orders and embeddings into $\mathbb{R}$ and of cycle order and the oriented circumference, but it also includes the exceptional case of the pre-betweenness induced by the non-oriented cycle of length four, $C_4$.

It is easy to see that axiom \textbf{B} is implied by axiom \textbf{B'} given by:
$$\textbf{B'}:\forall a,b\in V, a\neq  b\implies\overrightarrow{ab}=V.$$
The converse is not necessarily true. However, it is clear that it is true when \textbf{A} or \textbf{E} also holds. Less obviously, we can obtain axiom \textbf{B'} from axiom \textbf{B} and axiom \textbf{F} given by:
$$\textbf{F}: \forall a,b,c\in V, abc\in\calb\implies bca \in\calb\lor cba \in\calb.$$

In fact, for the sake of a contradiction, let us assume that there are $a$ and $b$ in $V$, such that $\overrightarrow{ab}\neq V$. Then there is $x\in V$ such that none of the triples $xab,axb,abx$ belongs to $\calb$. From \textbf{B} we get that $xba,bxa$ or $bax$ belongs to $\calb$.  But from \textbf{F} we get the contradiction
$axb\in\calb\lor xab\in\calb$, when $bax\in\calb$,
the contradiction
$xab\in \lor axb\in\calb$, when $bxa\in\calb$, and the contradiction
$bax\in\calb\lor abx\in\calb$, when $xba\in\calb$. 

Notice that axiom \textbf{F} is not implied by axioms \textbf{C} and \textbf{B'}, since the ternary relation 
$\{abc,bac,cab\}$, does satisfy axioms \textbf{C} and \textbf{B'}, but not axiom \textbf{F}. However, from axioms
\textbf{C'} and  \textbf{B'} we obtain \textbf{F}. In fact, given $a,b$ and $c$ in $V$ such that $abc\in\calb$, then axiom \textbf{C'} implies that $bac\notin\calb$. From axiom \textbf{B'} applied to $\overrightarrow{ba}$ we get that one of the triples $bca$ or $cba$ belongs to $\calb$. 

In this work we are primarily interested in relations satisfying \textbf{B}, \textbf{C}, \textbf{D}, \textbf{F}, \textbf{2}. Our motivation comes from ternary relations  $\calb_d$ defined in terms of a quasimetric $d$, a function that satisfies the same properties than a metric, with the exception of symmetry. Such ternary relations satisfy axioms \textbf{C}, \textbf{D}, \textbf{2} and \textbf{3}. In the seek of a characterization of ternary relations  $\calb_d$ defining exactly three or four lines none of which is $V$, for arbitrary sets $V$  \cite{gamaza2024}, we have found that they come from \emph{an expansion} of a ternary relation also associated to a quasimetric, but on three or four points, respectively, satisfying \textbf{B'}.

The first such example was obtained in \cite{armaza2023} and consists of a quasimetric on four points definining exactly three lines none of them being the set of the four points.  Before that, the mere existence of such quasimetric $d$ was unclear. In fact, in 2008 Chen and Chv\'{a}tal \cite{chenchvatal} conjectured that these spaces do not exist, when $d$ is a metric. More precisely, the conjecture says that the number of lines defined by a ternary relation  $\calb_d$ on a set $V$, when $d$ is a metric,  is at least the number of points of $V$, unless $V$ itself is a line.  This conjecture is still wide open, although there have been some recent interesting advances, mainly for metrics coming from several families of graphs (for example, \cite{abbemaza, abkha, beboch, bekharo, ch2014, kantor, jualrove}). Moreover, asymptotic bounds for the number of lines in terms of the number of points in $V$ have been proven, under the assumption that the set $V$ is not a line. The current best upper bound, obtained by Huang \cite{huang}, is $\Omega(n^{2/3})$, where $n$ is the number of points. It is worth to mention that the conjecture has been proved for the class of quasimetrics defined by semi-complete directed graphs \cite{arma2021,armaza2023sc} and it may still be valid for quasimetrics defined by directed graphs, as the quasimetrics discovered in \cite{armaza2023sc,gamaza2024}, mentioned above, cannot be defined by digraphs.

We also consider systems satisfying axioms \textbf{B}, \textbf{C},  \textbf{D}, \textbf{F} and \textbf{9}, where $$\textbf{9}: \forall a,b,c,x\in V, abc\in\calb\implies abx\in\calb\lor xbc.$$  The interest on axiom \textbf{9} comes from the fact that replacing axioms \textbf{1} and \textbf{2}, or axiom \textbf{2}, by this axiom, results in two sets of axioms, one for a betweenness and one for a cycle order, respectively \cite{huntkline,hunt1935}.

For these systems we prove that, if for some triple $abc\in\calb$ we also have $cab\in\calb$, then the system is a cycle order. Therefore, by including axiom \textbf{9} instead of axiom \textbf{2}, while keeping axiom \textbf{F}, we obtain precisely two type of systems: pre-betweennessess or cycle orders. However, systems satisfying $\{\textbf{B}, \textbf{C},\textbf{D}, \textbf{F},\textbf{2}\}$, additionally to pre-betweennesses and cycle orders include more complex structures. Our main result is a complete description of these systems. Roughly speaking, they are obtained from a cycle order by duplicating points in a way that such a point and its duplicate have all the remaining points in between them.

\section{Main result}

In this section, we present an independent set of axioms of ternary relations with a unique line. As mentioned above, when the ternary relation is a betweenness or a cycle order, then it has only one line $V$.

To ease the presentation of our main result we collect in the following proposition several properties that are implied by axiom \textbf{B}, \textbf{C} \textbf{D} and \textbf{F}, and also some that comes from these axioms together with axiom \textbf{2}. 
In the rest of this work we say that a triple $abc$, for points $a,b$ and $c$ in a set $V$ endowed with a ternary relation $\calb$ is \emph{valid}, when it belongs to the ternary relation under consideration. 

\begin{proposition}
\label{p:tools}
	Let $\calb$ be a ternary relation on $V$ satisfying axioms \textbf{B},\textbf{C},\textbf{D} and \textbf{F}. Then, we have the following properties:
    
	\begin{enumerate}
		\item[\textbf{L}\,] $: \forall a,b,c \in V, abc\in\calb\implies \calb\cap \{a,b,c\}^3\in \{ \{abc,cba\}, \{abc,bca,cab\} \}.$ 
        
		\item[\textbf{G}\,]$: \forall a,b,c,d \in V, abc\notin\calb\implies bac \in\calb\lor acb\in\calb.$
	        
	    \item[\textbf{C'}]$: \forall a,b,c \in V, abc\in\calb\implies bac\notin\calb$
	\end{enumerate}
	
	Moreover, if \textbf{2} holds, then we also have the following properties:
    
	\begin{enumerate}
		\item[\textbf{3}\ ]$: \forall a,b,c,x \in V, abc\in\calb\land bxc \in\calb\implies axc\in\calb$.
	    \item[\textbf{2'}]$: \forall a,b,c,x \in V, abc\in\calb\land axb \in\calb\implies axc\in\calb$.
	    \item[\textbf{3'}]$: \forall a,b,c,x \in V, abc\in\calb\land axb \in\calb\implies xbc\in\calb$.

	\end{enumerate}
\end{proposition}

\begin{proof}
	To ease the presentation, along the proof we do not mention explicitely the use of axiom \textbf{D}. Let us assumme that \textbf{B}, \textbf{C}, \textbf{D} and \textbf{F}. 

	\begin{enumerate}
		\item[\textbf{L}\,] As $abc\in\calb$, by \textbf{F} we have that $cba\in\calb\lor bca\in\calb$. Assumme $cba\in\calb$.   By \textbf{C} we get that $cab\notin\calb\land acb\notin\calb$. Now, we have to prove that $ bca\notin\calb\land bac\notin\calb$. Suppose that $bca\in\calb$, by \textbf{F} we get $cab\in\calb\lor acb\in\calb$, which is a contradiction. Thus, $bac\in\calb$.   By \textbf{F}, $acb\in\calb\lor cab\in\calb$, which is also a contradiction. Then  $ \calb\cap \{a,b,c\}^3=\{abc,cba\}$.
	
		Now, assume $bca\in\calb$.   By \textbf{C}, we have that $ bac\notin\calb\land acb\notin\calb$.   By \textbf{F}, we have that $acb\in\calb\lor cab\in\calb$, since $acb\notin\calb$ then $cab\in\calb$ and $ cba\notin\calb$. Thus, $ \calb\cap\{a,b,c\}^3=\{abc,bca,cab\}$.
        
		\item[\textbf{G}\,] Assume that $ abc\notin V \land bac\notin V\land acb\notin\calb$.   By $\textbf{B}$ we get that $bca\in\calb\lor cab\in\calb\lor cba\in\calb$ and by \textbf{L} we know that $| \calb\cap \{a,b,c\}^3|=3$. Thus, $ \calb\cap \{a,b,c\}^3=\{bca,cab,cba\}$ which is a contradiction, as $cab\in\calb\implies cba\notin\calb$, by \textbf{C}.
	 
		\item[\textbf{C'}\,] Since $abc\in\calb$, by \textbf{L}  we get that $ \calb\cap \{a,b,c\}^3$ is equal to  $\{abc,cba\}$ or to $\{abc,bca,cab\}$. Either case, $bac\notin\calb$.
        	\end{enumerate}

	Now, assume that \textbf{2} holds.
	\begin{enumerate}
		\item[\textbf{3}\,] We want to prove $axc\in\calb$.   By \textbf{2}, $abc\in\calb\land bxc\in\calb\Rightarrow abx\in\calb$.  By \textbf{L} we know that the set of valid triples of the set $\{b,c,x\}$ is either $\{bxc,xcb,cbx\}$ or $\{bxc,cxb\}$. 
	
		In the first case, by \textbf{2}, we have that $acx\in\calb\land cbx\in\calb \Rightarrow acb\in\calb$ which contradicts $abc\in\calb$. Hence, $acx\notin\calb$ and, by \textbf{G} we get that $cax\in\calb\lor axc\in\calb$. If $cax\in\calb$, then, by \textbf{2} applied to $cax\in\calb\land abx\in\calb$, we get that $cab\in\calb$. Again, by \textbf{2} and $xcb\in\calb\land cab\in\calb$, we get $xca\in\calb$. From \textbf{L} applied to $\{a,c,x\}$ we conclude $axc\in\calb$. 
	
		In the second case, we first prove that $ bxa\notin\calb$. Otherwise, from \textbf{L}  we  also get $xab\in\calb$. By \textbf{2} applied to $cxb\in\calb\land xab\in\calb$ we get $cxa\in\calb$. Again by \textbf{2} applied to $bca\in\calb\land cxa\in\calb$ and to $cba\in\calb\land bxa\in\calb$ we get $bcx\in\calb$ and $cbx\in\calb$, respectively. Thus, we get  $bca\notin\calb\land cba\notin\calb$ which contradicts \textbf{L} applied to $\{a,b,c\}$.
	
		Since $bxa\notin\calb$, by \textbf{L} applied to $\{a,b,x\}$ we get $xba\in\calb$.   By \textbf{2}, we have that $bxc\in\calb\land xac\in\calb\Rightarrow bxa\in\calb$ which is a contradiction with  \textbf{L} applied to $\{a,b,x\}$. Thus, we have $xac\notin\calb$ and by \textbf{C'} we obtain $axc\in\calb\lor xca\in\calb$. When $xca\in\calb$ we can apply \textbf{2} to $xca\in\calb\land cba\in\calb$ to get $xcb\in  \calb$, which is contradiction. Then, we have $xca\notin  \calb$ and thus the desired conclusion $axc\in  \calb$ holds.
	
		\item[\textbf{2'}\,] Let $abc\in  \calb\land axb\in  \calb$. Assume that $axc \notin  \calb$.   By \textbf{G} we have that $xac\in  \calb$ or $acx\in  \calb$.   By \textbf{3}, $xac\in  \calb\land abc\in  \calb\implies xab\in  \calb$ which contradicts $axb\in  \calb$, by \textbf{C'}.
	
		Thus, we have that $xac\notin\calb\land acx\in  \calb$.   By \textbf{L} we get that $xca\in  \calb$.   By \textbf{2},  $xca\in  \calb\land cba\in\calb\implies xba\in  \calb$. Again by \textbf{L} we get that $bxa\notin  \calb$ and thus $bax\in  \calb$.   By \textbf{3} we get $bax\in  \calb\land acx\in  \calb\implies bac\in  \calb$ which is a contradiction with $abc\in  \calb$, by \textbf{C'}.
	
		\item[\textbf{3'}]  We prove \textbf{3'}. Assume $ xbc \notin  \calb$. Then \textbf{G} implies $bxc\in  \calb\lor xcb\in  \calb$. In the first case, by \textbf{3}, $abc\in  \calb\land bxc\in  \calb\implies abx\in  \calb$, which is a contradiction with $axb\in  \calb$. Thus, $xcb\in  \calb$.   By \textbf{2} we have that $axb\in  \calb\land xcb\in  \calb\implies acb\in  \calb$ which is a contradiction with $abc\in  \calb$.

	\end{enumerate}
 \end{proof}

For further reference, we say that a 3-set $U$ is \emph{symmetric} if $U^3\cap \calb=\{xyz,zyx\}$, for $\{x,y,z\}=U$. Hence, a 3-set $\{a,b,c\}$ is symmetric if and only if $$U^3\cap \calb \in \{ \{abc,cba\},\{acb,bca\}, \{bac,cab\}\}.$$ Similarly, we say that a 3-set $U$ is \emph{cyclic} if $U^3\cap \calb=\{xyz,yzx,zxy\}$, for $\{x,y,z\}=U$. Hence, a 3-set $\{a,b,c\}$ is cyclic if and only if $$U^3\cap \calb \in \{ \{abc,bca,cab\},\{acb,cba,bac\}\}.$$ By using these definitions we have that \textbf{L} in previous proposition says that 3-sets are either cyclic or symmetric, but not both, when axioms \textbf{B}, \textbf{C}, \textbf{D} and \textbf{F} are valid.

Since $$\textbf{A} \lor \textbf{E} \implies \textbf{F}$$ then, the sets of axioms $\{ \textbf{B}, \textbf{C}, \textbf{D}, \textbf{F}, \textbf{2}\}$ and $\{ \textbf{B}, \textbf{C}, \textbf{D}, \textbf{F}, \textbf{9}\}$ are independent. For the latter set of axioms  we have the following result.

\begin{theorem}
	The set of axioms \{\textbf{B},\textbf{C},\textbf{D},\textbf{F},\textbf{9}\} implies $\textbf{A}\lor \textbf{E}$. 
\end{theorem}

\begin{proof}
	From Proposition~\ref{p:tools} we know that \textbf{C'} holds. We now prove that \textbf{2} follows from \textbf{C'} and \textbf{9}. In fact, if $abc\in\calb\land bxc\in  \calb$, then $abx\in  \calb\lor xbc\in  \calb$. From \textbf{C'} we get that $xbc\notin  \calb$ which proves that $abx\in  \calb$.
	We can now use any of the properties proven in Proposition \ref{p:tools}.
	
	We may assume that there are $a,b,c \in V$ such that $ \calb\cap \{a,b,c\}^3=\{abc,bca,cab\}$, as otherwise, \textbf{A} holds and then $\calb$ is a betweenness. Let $x\in V\setminus\{a,b,c\}$. We will prove that any set $\{u,v,w\}\subseteq\{a,b,c,x\}$ is cyclic.  
	
	By \textbf{9}, we have the following $$(abx\in \calb \lor xbc\in \calb)\land (bcx\in \calb\lor xca\in \calb)\land (cax\in \calb\lor xab\in 
	\calb).$$ Without loss of generality, we may assume that $\{a,b,x\}$ is symmetric.
	
	If $abx$ and $xba$ belong to $\calb$, then $xab\notin \calb$ and thus $cax\in \calb$. By \textbf{9} we get $xbc\in \calb$ or $cba\in \calb$. Hence, $xbc\in \calb$. From $xba\in \calb$ and $bca\in \calb$ we get $xca\in \calb$, which together with  $cax\in \calb$ implies $axc\in \calb$. From $axc\in \calb$ and $xbc\in \calb$ we get that $axb\in \calb$, a contradiction. Similarly, if $bax$ and $xab$ belong to $\calb$, then $abx\notin \calb$ and thus, $xbc\in \calb$. By \textbf{9} we get $bac\in \calb$ or $cax\in \calb$. Hence, $cax\in \calb$. From $bax\in \calb$ and $bca\in \calb$ we get $bcx\in \calb$, which together with  $xbc\in \calb$ implies $cxb\in \calb$. From $cax$ and $cxb$ in $\calb$ we obtain that $axb\in \calb$, a contradiction.
	
	Now, consider the case that the triples $axb$ and $bxa$ belong to $\calb$. By \textbf{9} we get $bxc\in \calb$ or $cxa\in \calb$ which shows that $bcx$ and $xca$ do not belong to $\calb$ by \textbf{C'}, a contradiction. Therefore, we get as conclusion that \textbf{E} holds.
\end{proof}

Before stating our main result, we first present some necessary definitions. For a given integer $p$, let $D_{p,p}$ be the digraph obtained from a directed cycle of length $p$ by duplicating all its vertices. More precisely, we have 

\begin{align*}
V(D_{p,p}) &= \{v_0,u_0,v_1,u_1,\ldots, v_{p-1},u_{p-1}\} \\
A(D_{p,p}) &= \{v_i v_{j} , v_i u_{j}, u_i v_{j}, u_i u_{j} \mid i, j \in \Z_p, j \equiv i+1 \mod p\}.
\end{align*}

For a digraph $D$, it is easy to see that any strongly connected  subdigraph of $D_{p,p}$ defines a quasimetric space with only one line (see Figure~\ref{fig:directed_graphs}).
\begin{figure}[ht]
    \centering
    \begin{subfigure}[t]{.3\textwidth}
        \centering
        \includegraphics[scale=1.5]{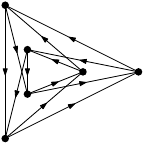}
    \end{subfigure}
    \begin{subfigure}[t]{.3\textwidth}
        \centering
        \includegraphics[scale=1.5]{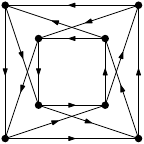}
    \end{subfigure}
    \begin{subfigure}[t]{.3\textwidth}
        \centering
        \includegraphics[scale=1.5]{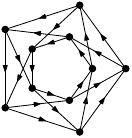}
    \end{subfigure}
    \caption{The directed graphs $D_{3,3}, D_{4,4}$ and $D_{5,5}$.}
    \label{fig:directed_graphs}
\end{figure}

One can see that the ternary relation $\calb_d$ associated to the quasimetric $d$ of $D_{p,p}$ defined by the shortest paths is given by:
$$\{w_iw_jw_k\mid w_l\in \{v_l,u_l\}, l\in \{i,j,k\}, ijk\in \calb_{\Z_p}\}\cup \{v_iv_ju_i,v_iu_ju_i, u_iv_jv_i,v_iu_ju_i\mid i,j\in \Z_p, i\neq j\},$$
where $ijk\in \calb_{\Z_p}$ if and only if $j \equiv i+l \mod p$, $k \equiv j+t \mod p$ and $i \equiv k+s \mod p$, for positive integers $l,t$ and $s$ such that $l+t+s=p$. That is, $\calb_{\Z_p}$ is the cycle order induced by the directed cycle of length $p$ whose set of vertex set is $\{0,1,\ldots,p-1\}$ and whose set of arcs is $\{(0,1),\ldots, (p-2,p-1),(p-1,0)\}$.

In the following theorem, we prove that any ternary relation $\calb$, over a finite set $V$, satisfying \textbf{B}, \textbf{C}, \textbf{D}, \textbf{F} and \textbf{2} which contains at least one cyclic 3-set is isomorphic to one defined by the quasimetric $d$ induced by a strongly connected subdigraphs of $D_{p,p}$, for some $p\leq |V|\leq 2p$.

\begin{theorem}
\label{t:type1iffD}
    If $\calb$ is a ternary relation on a set $V$ satisfying axioms \{\textbf{B}, \textbf{C}, \textbf{D}, \textbf{F}, \textbf{2}\}, then there is $p\leq |V|$ such that $\calb$ is isomorphic with $\calb_d$, for $d$ the quasimetric of a strongly connected subdigraph of $D_{p,p}$ unless $\calb$ also satisfies \textbf{A}, in which case, $d$ is the metric of a path on $V$ or of a undirected cycle of length four.
\end{theorem}

\begin{proof}
	First, consider the case where $\calb$ has no cyclic 3-set. Then, from \cite{shepperd1956transitivities} we know that $\calb=\calb_d$, for $d$ the metric defined by a path or by an undirected cycle of length four. So, we focus on the case where $\calb$ has a cyclic 3-set. 
	
	To ease the presentation, let us denote by $[xy]$ the set of all points $z\in Q$ such that the triple $xzy$ belongs to $\calb$ or $z\in \{x,y\}$.

	Let $C: x_0,\ldots,x_{p-1}$ be a maximal geodesic circuit of $\calb$. That is,  $x_{i-1}x_ix_{i+1}\in \calb$, for each $i=0,\ldots,p-1$ (note that the arithmetic in the indices is carried out in $\Z_{p}$) and for each $a \notin C$, $x_iax_{i+1}\notin \calb$, for each  $i=0,\ldots,p$. Also, it is worth to notice that this implies that for each $i\in \Z_p$ and for each $s\notin \{i,i+1\}$ the triples $x_{i}x_{i+1}x_s$ and $x_sx_ix_{i+1}$ are valid, while the triples $x_{i+1}x_{i}x_s$ and $x_sx_{i+1}x_{i}$ are not valid.
	
	If $C = V$, then $p = n$ and so $\calb$ is isomorphic to a directed cycle of on $n$ vertices. So, assume that $V \setminus C \neq \emptyset$.
	
	For $a \in V \setminus C$, let $j$ and $k$ be such that $[x_j a] \cap C = \{ x_j \}$ and $[ a x_k ] \cap C = \{x_k\}$. Since  $a \in \overrightarrow{x_j x_k}$, we get that $x_j a x_k\in \calb$, thus $j\neq k$. Moreover, $k\neq j+1$, as $C$ is maximal. 
    For each $l\in \Z_p$ such that $x_jx_lx_k$ belongs to $\calb$ we have that $x_jx_la$ and $x_jax_l$ do not belong to $\calb$ by the choice of $j$ and, the choice of $k$ and \textbf{3'}, respectively.

    By \textbf{L} we know that $ax_lx_j\notin \calb$, as otherwise, either $x_jx_la$ or $x_jax_l$ belongs to $\calb$. Similarly, $x_lax_j\notin \calb$. Therefore, by \textbf{B} we get that $ax_jx_l$ and $x_lx_ja$ belong to $\calb$. Similarly, we can prove that $ax_kx_l$ and $x_lx_ka$ belong to $\calb$.

    If $k\neq j+2$, then we have that $x_jx_{j+1}x_{j+2}$ and $x_{j+1}x_{j+2}x_k$ are valid and $p\geq 4$. We prove in this case that $a\notin \overrightarrow{x_{j+2}x_{j+1}}$, contradicting \textbf{B'}. From previous reasoning applied to $l=j+2$ we have that $ax_jx_{j+2},x_{j+2}x_ja\in \calb$. From this we get that $ax_jx_{j+1},ax_{j+1}x_{j+2}$ belongs to $\calb$. In turns, this implies that $ax_{j+2}x_{j+1}\notin \calb$. We also know that $x_{j+1}x_ka$ is valid which implies that $x_{j+1}x_{j+2}a$ is valid and thus $x_{j+2}x_{j+1}a$ is not valid. Finally, if $x_{j+2}ax_{j+1}$ belongs to $\calb$, then since $ax_jx_{j+1}$ is valid we get that $x_{j+2}ax_j$ is valid, as $x_{j+2}x_ja\in \calb$, we get a contradiction with \textbf{C}.
    
	Thus, we have that $k=j+2$. Therefore, we conclude that for each $a\notin C$, there is $i \in \Z_p$, such that $x_{i-1}ax_{i+1}$ is valid. Moreover, $\{x_{i},x_l,a\}^3\cap \calb=\{ax_lx_i,x_ix_la\}$, for $l\in \{i-1,i+1\}$. We show that this is also valid for $l \in \Z_p \setminus \{i-1,i,i+1\}$. That is,     $$[ax_i]\cap[x_ia]=V.$$ From $x_lx_ix_{i+1}\in \calb$ we get that $x_lax_i\notin \calb$, as otherwise, $ax_ix_{i+1}\in \calb$. Similarly, since $ ax_ix_{i+1}$ and $x_{i-1}x_ia$ do not belong to $\calb$, from $x_lx_{i-1}x_{i}\in \calb$ we get that $x_lx_ia\notin \calb$, from $x_ix_{i+1}x_l\in \calb$ we get that $ax_ix_{l}\notin \calb$, and from $x_{i-1}x_{i}x_l\in \calb$ we get that $x_iax_{l}\notin \calb$.  Therefore, $[x_ia]\cap [ax_i]=V$.

    This implies that, for $l\neq i$, $x_{l-1}ax_{l+1}\notin \calb$. In fact, for $i\notin \{l-1,l,l+1\}$, we know that  $x_{i-1}x_{l-1}x_{l+1}\in \calb$. Since the choice of $i-1$ implies that $x_{i-1}x_ia\notin \calb$, we get that $x_{l-1}ax_{l+1}\notin \calb$. For $i=l-1$, the choice of $i+1$  implies that $ax_ix_{i+1}\notin \calb$ which, together with the fact that $x_{i-2}x_ix_{i+1}\in \calb$ shows that $x_{i-2}ax_i\notin \calb$. For $i=l+1$, the fact that $x_{i-2}x_{i}x_{i+2}\in \calb$ implies $x_{i-2}ax_i\notin \calb$, as $ax_ix_{i+1}\notin \calb$.

    We now prove that if $x_ix_lx_t\in \calb$, then $ax_lx_t, x_lx_ta$ and $x_tax_l$ belong to $\calb$. From $x_ix_ta\in \calb$ we get that  $x_lax_t\notin \calb$. Similarly, from $x_ix_ta\in \calb$ we get that $x_tx_la\notin \calb$. Thus, $\{a,x_l,x_t\}^3\cap \calb=\{ax_lx_t,x_lx_ta, x_tax_l\}$. Since $\{x_i,x_l,x_t\}$ is cyclic, the same conclusion can be obtained when $x_lx_tx_i$ or $x_tx_ix_l$ belongs to $\calb$. This implies that $C'=(C\setminus \{x_i\})\cup \{a\}$ is also a maximal geodesic circuit. We prove now that if there is $b\notin C'$ with $x_{i-1}bx_{i+1}\in \calb$, then $b=x_i$. For a contradiction, let us assume that $b\neq x_i$. From previous argument we have that $x_{i-1}\in [ax_i]\cap [bx_i]\cap [ab]$. By \textbf{L}, one of the triples $abx_i,bax_i$ or $bx_ia$ must be valid. From $ax_{i-1}b\in\calb$, we get that $x_{i-1}bx_i\in \calb$, when $abx_i\in \calb$, which contradicts $x_{i-1}\in [bx_i]$. Similarly, from $bax_i\in \calb$  and  $bx_{i-1}a\in\calb$ we get the contradiction $x_{i-1}\notin [ax_i]$. Finally, when $bx_ia$ is in $\calb$ we get the contradiction $x_{i-1}x_ia\in \calb$, as $bx_{i-1}x_i\in \calb$. 
    
	From now on, when for a point $a$ not in $C$ we have that $x_{i-1}ax_{i+1}$ belongs to $\calb$, then we denote it by $y_i$. We have just proven that $y_i$ and $x_i$ behave interchangeably in $\calb$. For each $i \in \Z_p$, we call these (pairs of) points, \textit{points of type} $i$ in $V$. We have now established the frame to present a complete description of $\calb$. 
    
    So far, we have seen that three points in $V$ define a valid triple $uvw$ in one of two ways. First, if $u, v$ and $w$ are of type $i$, $j$ and $k$, respectively, with $x_ix_jx_k\in \calb$. On the other hand, if $u$ and $w$ are of the same type, and consequently, $v$ has a different type and $uvw$ and $wvu$ belong to $\calb$.

    Let $I\subseteq \Z_p$ the set of indices $i$ such that there is $u\in V\setminus C$ with $x_{i-1}ux_{i+1}\in \calb$. We associated to $(V,\calb)$ the subdigraph $D^*$ of $D_{p,p}$  containing a directed cycle of length $p$ with set of vertices 
    $$\{v_0,\ldots,v_{p-1}\}\cup \{u_i \mid i\in I\}.$$ To finish the proof we show that the function $f: V \rightarrow V(D^*)$ given by: $f(x_i) = v_i$, for each $i\in \Z_p$, and $f(y_i)=u_i$, for each $i\in I$, defines an isomorphism between $\calb$ and the ternary relation $\calb_d$ induced by the quasimetric $d$ of $D^*$. 
    
    Given our description of $\calb$ we have presented  above, the function $f$ is well-defined and it is easy to see that it is a bijection between $V$ and $V(D^*)$. To finish the proof we show that $f$ is an isomorphism between $\calb$ and $\calb_d$.  In fact, let  $a_ia_ja_k$ be a valid triple for $\calb$. We prove that $f(a_i)f(a_j)f(a_k)$ is a valid triple for $\calb_d$. If $a_i$, $a_j$ and $a_k$ are of type $i$, $j$ and $k$, respectively, and $x_ix_jx_k\in \calb$, then we have that $f(a_l) \in \{v_l,u_l\}$, for each $l\in \{i,j,k\}$ and $f(a_i)f(a_j)f(v_k)$ is valid for $\calb_d$. Otherwise, $i=k$, $j \notin \{i,k\}$ and so, $\{f(a_i),f(a_k)\}=\{v_i,u_i\}$ and $f(a_j)\in [v_iu_i]\cap [u_iv_i]$. 
    
\end{proof}

\section{Conclusion}
The results we have found suggest that a similar version could be valid for arbitrary sets $V$, where for a maximal cycle order $C$ on $V$, any point $x$ outside $C$ could be associated to a unique point $x'$ in $C$ such $(C\setminus \{x'\})\cup \{x\}$ is also a maximal cycle order and $[xx']=V$. We leave it as an open question.

\bibliographystyle{plain}
\bibliography{sample}

\end{document}